\documentclass[12pt]{amsart}
\newcommand{\nn}{\mathbb{N}}

\newcommand{\real}{\mathop{\mathrm{Re}}}
\newcommand{\<}{\langle}
\renewcommand{\>}{\rangle}

\newtheorem{theorem}{Theorem}[section]

\newtheorem{lemma}[theorem]{Lemma}

\theoremstyle{definition}

\numberwithin{equation}{section}

\begin{document}
\title[Continuity of CP-semigroups]{Continuity of CP-semigroups in the
point-strong operator topology}
\author{Daniel Markiewicz}
\address{Department of Mathematics, University of Toronto,
 40 St. George Street, Rm 6290, Toronto, Ontario M5S 2E4, Canada.}
\email{dmarkiew@fields.utoronto.ca}
\author{Orr Moshe Shalit}
\address{Department of Mathematics, Technion - Israel
Institute of Technology, 32000, Haifa, Israel.}
\email{orrms@techunix.technion.ac.il}
\thanks{D.M. was partially supported by the Fields Institute during
the Thematic Program on Operator Algebras. O.M.S. was partially
supported by the Gutwirth Fellowship.}

\date{October 3, 2007}
\begin{abstract}
We prove that if $\{\phi_t\}_{t \geq 0}$ is a CP-semigroup acting on
a von Neumann algebra $M \subseteq B(H)$, then for every $A\in M$ and $\xi \in H$,
the map $t \mapsto \phi_t(A)\xi$ is norm-continuous. We discuss the
implications of this fact to the existence of dilations of CP-semigroups to
semigroups of endomorphisms.
\end{abstract}
\maketitle

\section{Introduction}
Let $H$ be a Hilbert space, not necessarily separable, and let $M \subseteq
B(H)$ be a von Neumann algebra. A \emph{CP-semigroup} on $M$ is a family $\phi =
\{\phi_t: M \to M\}_{t\geq0}$ of contractive normal completely positive
maps which satisfies the following properties:
\begin{enumerate}
 \item $\phi_0(A) = A$, $\forall A \in M$
 \item $\phi_{s+t} = \phi_s \circ \phi_t$, $s,t\geq 0$
 \item for all $A \in M$ and $\omega \in M_*$, $\lim_{t\to t_0}
\omega(\phi_t(A)) = \omega(\phi_{t_0}(A))$
\end{enumerate}
where $M_*$ denotes the predual of $M$. We shall refer to continuity condition
(3) as \emph{continuity in the point-$\sigma$-weak
topology}. It is equivalent to \emph{continuity in the point-weak
operator topology}, i.e.
$$
\lim_{t\to t_0} \< \phi_t(A)\xi,\eta \> = \<
\phi_{t_0}(A)\xi, \eta \>, \quad A\in M, \xi, \eta \in H.
$$
A CP-semigroup $\phi$ is called an \emph{E-semigroup} if $\phi_t$ is a
$*$-endomorphism for all $t\geq 0$.

In this note we prove that CP-semigroups satisfy a seemingly stronger continuity condition, namely
\begin{equation}\label{eq:STOP}
\lim_{t\rightarrow t_0} \|\phi_t(A)\xi -\phi_{t_0}(A)\xi \| = 0 ,
\end{equation}
for all $A\in M, \xi \in H$. A semigroup satisfying (\ref{eq:STOP}) will be said
to be \emph{continuous in the point-strong operator topology}.
The proposition that CP-semigroups are continuous in the point-strong operator
topology has appeared in the literature earlier, but the proofs that
are available seem to
be incomplete. In the proofs of which we are aware, only 
continuity \emph{from the right} in the point-strong operator topology is
established. By this we mean that (\ref{eq:STOP}) holds for limits taken with
$t\searrow t_0$.

We consider the continuity of CP-semigroups in the point-strong
operator topology to be an important property, because it plays a
crucial role in the existence of dilations of CP-semigroups to
E-semigroups. We are aware of five different proofs for
the fact that every CP-semigroup has a dilation to an E-semigroup:
Bhat~\cite{Bhat1996}, Selegue~\cite{Selegue1997},
Bhat--Skeide~\cite{BhatSkeide2000}, Muhly--Solel~\cite{MuhlySolel2002} and
Arveson~\cite{Arveson2003} (some of the authors require some
additional conditions, notably that the CP-semigroup be unital or
that the Hilbert space be separable). In order to show that the
minimal dilation of a CP-semigroup to an E-semigroup is continuous
in the point-weak operator topology, all authors make use of
continuity of the CP-semigroup in the point-strong operator topology, either
implicitly or explicitly. 

This research was conducted while the first named author was at the Fields
Institute. He thanks George A. Elliott for his generous hospitality and support. The second 
named author wishes to thank Baruch Solel, his thesis advisor, for his guidance and encouragement.

\section{Preliminaries}\label{sec:preliminaries}

Let $M$ be a von Neumann algebra acting on a Hilbert space, which is not
assumed to be separable. Let $\phi = \{\phi_t : M \to M \}_{t\geq 0}$ be a
CP-semigroup acting on $M$. We denote by $M_*$ the set of
$\sigma$-weakly continuous linear functionals on $M$. We shall denote by
$\sigma(M_*,M)$ the pointwise convergence topology of $M_*$ as a subset of the
dual space of $M$.

Let $\delta$ be the generator of $\phi$, and let $D(\delta)$ denote its
domain:
$$
D(\delta) = \{ A \in M : \exists \delta(A) \in M\,  \forall \omega \in
M_* \, \lim_{t\to 0+} t^{-1} \omega(\phi_t(A) - A)=\omega(\delta(A))  \}.
$$

\begin{lemma}\label{lemma:right-cts} For every $A\in M$ and $\xi \in H$, the map
$t \mapsto \phi_t(A)\xi$ is continuous from the right (in norm).
\end{lemma}

The proof of this result can be found in the literature, for example as
Lemma A.1 of \cite{Bhat2001} or Proposition 4.1 item 1 in
\cite{MuhlySolel2002}. For completeness, let us present the argument from
\cite{Bhat2001}. 
Let $A \in M$, $\xi \in H$ and $t \geq 0$. For all $h>0$, we have, using the
Schwartz inequality for completely positive maps,
\begin{multline*}
\|\phi_{t+h}(A)\xi - \phi_{t}(A)\xi\| = \\
= \< \phi_{t+h}(A)^* \phi_{t+h}(A)
\xi, \xi \> - 2 \real \< \phi_{t+h}(A) \xi, \phi_{t}(A) \xi \> +
\|\phi_{t}(A)\xi \|^2 \hphantom{XXX} \\
\leq  \< \phi_h(\phi_{t}(A)^* \phi_{t}(A)) \xi, \xi \> - 2 \real \<
\phi_{t+h}(A) \xi, \phi_{t}(A) \xi \> + \|\phi_{t}(A)\xi \|^2 
\xrightarrow[h \to 0]{} 0.
\end{multline*}

We remark, however, that two-sided continuity
does not follow directly from continuity from the right. This is in contrast
with the situation of the classical theory of $C_0$-semigroups on Banach spaces
(see for example \cite{HillePhillips1974}). If $T = \{T_t\}_{t\geq 0}$ is a
semigroup of contractions on a Banach space $X$ such that the maps
$$
t \mapsto T_t (x)
$$
are continuous from the right in norm for all $x \in X$, then it is easy to show
that these maps are also continuous from the left in
norm\footnote{
for given $x \in X$, $0\leq t \leq a$, $\| T_{a-t}(x) - T_a(x) \| = \|
T_{a-t} (x - T_t(x)) \| \leq \| x - T_t(x)  \|.$
}.
In fact, when $X$ is separable, for
instance, it can be proven by measurability and integrability techniques that if
the  maps $t \mapsto f(T_t (x))$ are measurable for all $x \in X$ and $f \in
X^*$, then the maps $t \mapsto T_t(x)$ are continuous in norm for $t > 0$. In
the case of CP-semigroups on von Neumann algebras, however, these techniques
seem to require considerable modification. We provide here an alternative
approach to the problem.

Recall that a function $g: [0, 1] \to
H$ is \emph{weakly measurable} if for all $\eta \in H$, the complex-valued
function
$g_{\eta}(t)=\<\eta, g(t) \>$ is measurable. We will say that the function $g$
is \emph{strongly measurable} if there exists a family of countably-valued
functions
(i.e. assuming a set of values which is at most countable)
converging Lebesgue almost everywhere to $g$. (For more details, see Definition
3.5.4, p.~72, and the surrounding discussion in \cite{HillePhillips1974}).

\begin{lemma}\label{lemma:bochner-int}
For all $\xi\in H, A\in M$, the function $f:[0,1] \to H$ given by $f(t)
= \phi_t(A)\xi$ is strongly measurable and Bochner integrable on the interval
$[0,1]$.
\end{lemma}
\begin{proof}  The function $f$ is weakly continuous, since
$\phi$ is continuous in the point-weak operator topology. In particular, it is
weakly measurable.  Furthermore, by Lemma \ref{lemma:right-cts}, the function
$f$ is continuous from the right in norm, hence it is separably valued (i.e.,
its range is contained in a separable subspace of $H$). By Theorem~3.5.3 of
\cite{HillePhillips1974}, the function $f$ is strongly measurable because it is
is weakly measurable and separably valued.

Thanks to Theorem 3.7.4, p.~80 of \cite{HillePhillips1974}, in order to show
that $f$ is Bochner integrable it is enough to show that $f$ is strongly
measurable and that
$$\int_0^1 \|f(t)\|dt < \infty.$$
The latter condition is easy to verify, as $t \mapsto \|f(t)\|$ is a
right-continuous, bounded function on $[0,1]$.
\end{proof}

We thank Michael Skeide for the idea to use the continuity of $f$ from the
right in order to avoid making the assumption that $H$ is separable.

\begin{lemma}\label{lemma:domain-density}
Let $A \in B(H)$ be positive. Then there exists a sequence $A_n \in
D(\delta)$ of positive operators such that $A_n \to A$ in the $\sigma$-strong*
topology.
\end{lemma}
\begin{proof}
Recall that the sequence
$$
A_n = n \int_0^{1/n} \phi_t(A) dt
$$
(integral taken in the $\sigma$-weak sense) converges in the
$\sigma$-weak topology to $A$. Furthermore $A_n \in D(\delta)$ and it is a
positive operator for each $n>0$ since $\phi_t$ is a CP map for all $t$. It is
easy to check that $\| A_n \| \leq \|A \|$ for all $n$ since $\phi_t$ is
contractive.

Now observe that for each $\xi \in H$, the map $t \mapsto \phi_t(A) \xi$ is
Bochner integrable on $[0,1]$ (see Lemma~\ref{lemma:bochner-int}), hence in
fact we have
$$
A_n \xi = n \int_0^{1/n} \phi_t(A)\xi dt
$$
where the integral is taken in the Bochner sense. The identity holds because
for all $\eta \in H$, $n\in \nn$ we have:
$$
\< A_n \xi , \eta \> = n \int_0^{1/n} \<\phi_t(A)\xi, \eta \> dt = \<
 n \int_0^{1/n} \phi_t(A)\xi dt, \eta \>.
$$

We now show that $A_n \to A$ strongly. Let $\xi \in H$ be fixed.
\begin{align*}
\|A\xi -  A_n \xi\| &= \|n \int_0^{1/n}A\xi dt  -
n \int_0^{1/n}\phi_t(A)\xi  dt \| \\
&\leq n \int_0^{1/n} \|A\xi - \phi_t(A)\xi \|dt  .
\end{align*}
The latter goes to zero by continuity from the
right (Lemma~\ref{lemma:right-cts}).
Since $A_n$, $A$ are positive operators, by considering adjoints we obtain that
$A_n \to A$ in the strong* topology. Finally, since the sequence is
bounded, we have convergence in the $\sigma$-strong* topology.
\end{proof}

\begin{lemma}\label{lemma:uniformity}
Let $A_n$ be a bounded sequence of operators in $M$ converging to $A$ in the
$\sigma$-strong* topology and let $t_0\geq0$. Then for every sequence $t_k \to
t_0$, $\xi \in H$ and $\epsilon>0$, there exists $N\in \nn$ such that for $n\geq
N$,
$$
\| \phi_{t_k}(A_n -A) \xi  \| < \epsilon, \quad \text{for all }k .
$$
\end{lemma}
\begin{proof}
Let $B_n=(A_n-A)^*(A_n-A)$, $\omega_k(X)= \< \phi_{t_k}(X) \xi, \xi \>$ and
$\omega(X)=\<  \phi_{t_0}(X) \xi, \xi \>$. Then we have that
$$
\| \phi_{t_k}(A_n - A)\xi  \|^2 = \< \phi_{t_k}(A_n -A)^*
\phi_{t_k}(A_n -A) \xi, \xi \> \leq \omega_k(B_n)
$$
since $\phi_t$ is a CP map for all $t$. Since $\phi$ is a point-$\sigma$-weakly
continuous semigroup, we have that $(\omega_k)$ is a
sequence of $\sigma$-weakly continuous linear functionals such that $\omega_k(X)
\to \omega(X)$ for all $X \in M$. Furthermore, $B_n$ is a bounded sequence
converging in the $\sigma$-strong* topology to 0. The latter holds because $A_n$
is a bounded sequence converging to $A$ in the $\sigma$-strong* topology and
multiplication is jointly continuous with respect to this topology in bounded
sets (of course * is also continuous). Finally, we obtain the desired
conclusion by applying Lemma III.5.5, p.151 of
\cite{Takesaki2002}, which states the following. Let $M$ be a von Neumann
algebra and let $\rho_k$ be a sequence in $M_*$ converging to $\rho_0 \in
M_*$ in the $\sigma(M_*, M)$ topology. If a bounded sequence $(a_n)$ converges
$\sigma$-strongly* to $0$, then $\lim_{n\to \infty} \rho_k(a_n) = 0$ uniformly
in  $k$.
%Notice that for $k=0$ the statement follows from
%normality of $\phi_{t_0}$.
\end{proof}

\section{The main result}
\begin{theorem}
Let $\phi$ be a CP-semigroup acting on a von Neumann algebra $M \subseteq B(H)$. Then for all
$\xi \in H$, $A \in M$ and $t_0 \geq 0$,
$$\lim_{t \rightarrow t_0}\|\phi_t(A)\xi - \phi_{t_0}(A)\xi\| = 0 .$$
\end{theorem}
\begin{proof}
Let $\epsilon >0$ be given, and let $(t_k)$ be a sequence converging to
$t_0$. By Lemma~\ref{lemma:domain-density}, there is a bounded sequence $(A_n)$
of operators $A_n \in D(\delta)$ such that $A_n
\to A$ in the $\sigma$-strong* topology. By Lemma~\ref{lemma:uniformity}, there
exists
$N \in \nn$ such that for $n \geq N$,
$$
\| \phi_{t_k}(A_n -A) \xi  \| < \frac{\epsilon}{3}, \quad \text{for all }
k\geq 0 .
$$
By an application of the Principle of Uniform Boundedness, if $X \in D(\delta)$
there exists $C_X > 0$ such that
$$
\sup_{s >0 } \frac{1}{s} \| \phi_s(X) - X \| \leq C_X <\infty  .
$$
Now notice that $A_n \in D(\delta)$ for all $n$, and in particular $\exists
C>0$ such that
$$
\sup_{s >0 } \frac{1}{s} \| \phi_s(A_N) - A_N \| \leq C  .
$$
Because $\phi_s$ is a contraction for all $s$, we obtain that for all $k$,
\begin{align*}
 \|  \phi_{t_k}(A_N)\xi - \phi_{t_0}(A_N)\xi \| & \leq \|  \phi_{t_k}(A_N) -
\phi_{t_0}(A_N) \|\; \| \xi \| \\
& \leq \|  \phi_{|t_k-t_0|}(A_N) -A_N \| \; \| \xi
\| \\
& \leq C \| \xi \| \; | t_k - t_0 | .
\end{align*}
In particular, we must have that $\| \phi_{t_k}(A_N)\xi
- \phi_{t_0}(A_N)\xi \| \to 0$ as $k \to \infty$. Thus there is $K\in \nn$
such that for $k\geq K$,
$$
\| \phi_{t_k}(A_N)\xi - \phi_{t_0}(A_N)\xi \| < \frac{\epsilon}{3} .
$$
We conclude that for $k \geq K$,
\begin{multline*}
\| \phi_{t_k}(A)\xi - \phi_{t_0}(A)\xi \| \leq
\| \phi_{t_k}(A -A_N) \xi  \| + \\
+ \| \phi_{t_k}(A_N)\xi - \phi_{t_0}(A_N)\xi \|  +
\| \phi_{t_0}(A_N -A) \xi  \| < \epsilon.
\end{multline*}
\end{proof}

\providecommand{\bysame}{\leavevmode\hbox to3em{\hrulefill}\thinspace}

\end{document}